\documentclass{article}
\usepackage{ed}
\newcommand{\col}[1]{ \begin{bmatrix} #1 \end{bmatrix} }
\newcommand{\kkk}{\Bbbk}
\newcommand{\vC}{\check{C}}
\newcommand{\dR}{\check{\Omega}}

\title{Residues in group completions and the \v{C}ech cohomology of $BG$}
\author{Edward Dewey}
\date{}
\begin{document}
\maketitle

\begin{abstract}
\noindent Let $G$ be a connected affine algebraic group over $\C$, $G \to
X$ be an open immersion of $G$-varieties, $Z = X-G$ and $i: Z
\to X$ be the inclusion.  Let $\alpha \in H^*(G,\C)$ be primitive.  We give a
method to compute the image of $\alpha$ in $H^*(Z, i^!\C_X)$, using a lift of
$\alpha$ along the first edge map of the \v{C}ech spectral sequence for
$H^*(BG, \C)$.  We apply it to the wonderful
compactification of a centerless semisimple group $G$.
\end{abstract}

\section*{Introduction}
Let $G$ be a connected affine algebraic group over $\C$, and let $X$ be an
\emph{equivariant partial completion} of $G$ - that is, a complex variety with
a left $G$-action and an open immersion $j: G \into X$ such that the action of
$G$ on $X$ restricts to the multiplication action of $G$ on itself.  Let
$Z = X-G$ and let $i: Z \into X$ be the inclusion.  We want to relate the Gysin map $H^*(Z/G,i^!\C_{X/G}) \to
H^*(X/G,\C)$ and the residue map $H^*(G,\C) \to H^{*+1}(Z,i^!\C_X)$.  
Our main result (\cref{resLemma}) does this in a special case, via the first edge map of the \v{C}ech spectral
sequence of $BG$. It is inspired by the degeneracy locus formula for Chern classes,
and can be used to recover a weak form of that formula (\cref{chern}).

\Cref{mainResults} explains \cref{resLemma}. 
\Cref{wonderful} applies it, in the case where $X$ is the wonderful
compactification of a centerless semisimple group $G$, to compute the residue
of a primitive element $\alpha \in H^*(G)$.  \Cref{proof} proves it. The theorem is more useful if you can compute the $1^{st}$
edge map. \Cref{edgeCalc} explains how to do that,
using Bott's approach to the Chern-Weil homomorphism.

There are many powerful theorems comparing the cohomology of a $G$-variety $X$
with its equivariant cohomology (\cite{GKM} contains many examples), and the
equivariant cohomology of $X$ was already used to compute a Gysin map in
\cite{loring}. 

\textbf{Acknowledgments:}  Thanks to Dima Arinkin and Andrei C\u{a}ld\u{a}raru
for some very useful conversations.  The author was supported by the NSF RTG
grant \#1502553.

\section{Residues and equivariant cohomology} \label{mainResults}

All schemes are over $\C$ and we always use the classical topology.  If $A$ is
a sheaf of abelian groups on $X$ there is an exact triangle $i_*i^!A \to A \to
j_*j\i A \to i_*i^!A[1]$ of sheaves on $X$.  The associated long exact sequence is the
\emph{Gysin sequence} \[\xymatrix{
\ldots \ar[r]& H^*(Z, i^!A) \ar[r]& H^*(X, A) \ar[r]& H^*(G, j\i A)
\ar[r]^{\res\hspace{2mm}}& H^{*+1}(Z, i^!A) \ar[r]& \ldots }\]
We call $\res: H^*(G, j^*A) \to  H^{*+1}(Z, i^!A)$ the \emph{residue map}.  For
$\alpha \in H^d(G, j\i A)$ we call $\res\,\alpha \in H^{d+1}(Z, i^!A)$ the
\emph{residue of $\alpha$ along $Z$}.

\begin{notation}
If $S$ is a scheme with left $G$ action, write $S/G$ for the stack quotient of
$S$ by $G$.  Write $BG$ for $pt/G$.  Write $S_\bul$ for the simplicial scheme
obtained as the nerve of the cover $S \to S/G$.  If $S = pt$ we will write
$BG_\bul$ instead of $pt_\bul$.  If $f: K \to S$ is a
$G$-equivariant map of schemes with $G$-action we write $f_p: K_p \to S_p$ for
the induced map.  We will often abuse notation and denote $f: K \to S$ and
its induced map $K/G \to S/G$ by the same letter.  

$\C_X$ is the constant sheaf on $X$ with stalk $\C$.  We are mostly interested
in cohomology with $\C$ coefficients and will write $H^*(X)$ for $H^*(X, \C)$.

If $E_\bul^{\bul,\bul}$ is a first-quadrant, cohomologically graded spectral
sequence abutting to a graded module $L^\bul$, write $L = F_0L \supset F_1L
\supset \ldots$ for the corresponding filtration and $\epsilon_n^{p,q}(E):
F_pL^{p+q} \to E_n^{p,q}$ the edge map.  We will often write
$\epsilon_n^{p,q}$, leaving $E$ implicit.
\end{notation}

\begin{definition}
There is a spectral functor $\vC_n^{p,q}(X_\bul, -)$ computing $H^{p+q}(X/G,-)$ with
$\vC_1^{p,q}(X_\bul,-) = H^q(X_p,-)$, called the \emph{\v{C}ech spectral sequence}
\cite[\href{http://stacks.math.columbia.edu/tag/06XJ}{Tag 06XJ}]{stacks-project}
\end{definition}

The differential on $\vC_1(X_\bul,-)$ is the alternating sum of
pullbacks $\sum (-1)^i \partial_i^*: H^q(X_p,\pi_p\i-) \to H^q(X_{p+1},
\pi_{p+1}\i-)$, where $\pi_p: X_p \to X/G$ is the natural map.

\begin{definition}
Call $\alpha \in H^*(G)$ \emph{primitive} if $\mu^*\alpha = \pi_0^*\alpha +
\pi_1^*\alpha$ where $\mu, \pi_0, \pi_1: G \times G \to G$ are the
multiplication and the two projections.  \end{definition}

Note that $\vC_2^{1,q}(BG_\bul,\C)$ is the subspace $H_{\mathrm{pr}}^q(G)
\subset H^q(G)$ consisting of the primitive elements.  $F_{p+1}H^{p+q}(BG)$ is the
kernel of $\epsilon_{n}^{p,q}$ for $0 \ll n$ and the codomain of
$\epsilon_n^{0,q}(\vC(BG_\bul, \C))$ is $0$ for $0 < n,q$.  Thus $F_1H^{q+1}(BG,\C) =
H^{q+1}(BG,\C)$ and the edge map $\epsilon_2^{1,q}$ is defined on all of
$H^{q+1}(BG,\C)$.  Schulman \cite{shulman} proved $\vC(BG_\bul, \C)$ degenerates on
page 2, so $\epsilon_2^{1,q}: H^{q+1}(BG) \to H_{\mathrm{pr}}^q(G)$ is
surjective.  Hopf proved that $H^*(G)$ is a wedge algebra on
$H^*_{\mathrm{pr}}(G)$ \cite[section 13]{samelson}. 

Consider the diagram 
\[\xymatrix{
Z \ar[r]   \ar[d]_{\pi_Z}  & X \ar[d]^\pi  & G \ar[d]\ar[l]   \\
Z/G \ar[r]                 & X/G           & pt   \ar[l]      }\]
As $\pi$ is smooth the diagram of Gysin sequences below commutes:
\[\xymatrix{
H^{q-1}(G) \ar[r] & H^q(Z, i^!\C_X) \ar[r]          & H^q(X) \ar[r]           & H^q(G)  \\
H^{q-1}(pt)\ar[r]\ar[u] & H^q(Z/G, i^!\C_{X/G}) \ar[r] \ar[u] & H^q(X/G) \ar[r] \ar[u]  & H^q(pt)\ar[u] }\]

For $q>0$, $H^q(pt)=0$ and there is an inverse map $H^q(Z/G, i^!\C) \gets
H^q(X/G)$.  In \cref{proof} we will prove our main result:
\theoremstyle{plain}
\newtheorem*{lemma:main}{Theorem \ref{resLemma}}
\begin{lemma:main}
Let $q>0$.  Then the diagram below commutes:
\[\xymatrix{
H^{q+1}(Z/G, i^!\C) \ar[d]^{\pi_Z^*}  & H^{q+1}(X/G) \ar[l] &
H^{q+1}(BG)\ar[d]^{\epsilon^{1,q}_2} \ar[l] \\
H^{q+1}(Z,   i^!\C)         &                           & H^{q}_{\mathrm{pr}}(G) \ar[ll]^{\res}
}\]
\end{lemma:main}

In order to apply this one wants to be able to compute
	$\epsilon_2^{1,q}(\vC(BG_\bul, \C))$. Below we will avoid thinking 
	hard about this by working only up to a scalar, but Bott has implicitly
	given an algorithm for it, which we describe \cref{edgeCalc}.

\begin{notation}
Write $x \sim y$ if $x = t y$ for some $t \in \C^*$.
\end{notation}

\begin{proposition}
\label{edgeProp}
Let $H^*(Y)_{red}$ be the subspace of reducible elements of the cohomology ring
of $Y$ under cup product.  Suppose $\dim H^{2d}(BG)/H^{2d}_{red}(BG)=1$, $\dim
H_{pr}^{2d-1}(G) = 1$, $\beta \in H^{2d}(BG)$ is irreducible and $\alpha \in
H_{pr}^{2d-1}(G)$.  Then $\epsilon_{2}^{1,2d-1}(\beta)\sim\alpha$. 
\end{proposition}
\begin{remark}
The conditions on $H^{2d}(BG)/H^{2d}_{red}(BG)$ and $H^{2d-1}_{pr}(G)$ are
satisfied whenever $G$ is semisimple and not of type $D_{2n}$.  They are also
satisfied if $G = GL_n$ or if $G$ is of type $D_{2n}$ and $d \neq n$ \cite[Section
3.7]{humph}.
\end{remark}
\begin{proof}
Since $\dim H_{pr}^{2d-1}(G)=1$ it suffices to check that
$\epsilon_2^{1,2d-1}(\beta)\neq 0$.  The cup product on $BG$ lifts to a bigraded multiplication on $\vC_2(BG_\bul,
\C)$, and $\vC_2(BG_\bul,\C)$ is generated as a ring by the $p=1$ row, that is
elements from $H^q(BG_1,\C)$. Therefore $H_{red}^{2d}(BG) \subset F_2H^*(BG) =
\ker \epsilon_2^{1,2d-1}$.  Since $\dim H^{2d}(BG)/H^{2d}_{red}(BG)=1$, if
$\epsilon_2^{1,2d-1}(\beta)=0$ then $\epsilon_2^{1,2d-1}$ is the zero map.  But
the \v{C}ech spectral sequence degenerates on page 2 so $\epsilon_2^{1,2d-1}$
is surjective, in particular nonzero.
\end{proof}

We would like to use \cref{resLemma} to compute $\res\, \alpha$ for $\alpha \in
H^*_{\mathrm{pr}}(G)$.  By ``compute'' we mean ``express in terms of the
intrinsic topology of $Z$'', ideally as a sum of products
of fundamental classes.  Our terminology is slightly nonstandard:

\begin{definition}
Let $i: K \into S$ be a closed immersion of smooth complex varieties of
codimension $r$.  The purity theorem gives a canonical isomorphism $i^! \C_S
\cong \C_K[-2r]$.  There is a corresponding isomorphism $H^0(K, \C) \to
H^{2r}(K, i^!\C_S)$.  We call the image of $1$ under this map the
\emph{fundamental class} of $K$, and denote it $[K]$.

If $S$ is smooth but $K$ is possibly singular, there is a unique element of $H^{2r}(K,
i^!\C_S)$ restricting to $[K] \in H^{2r}(K^{sm}, (i^{sm})^! \C_{S^{sm}})$ where $K^{sm}$
is the smooth locus of $K$ and $i: K^{sm} \to S-(K-K^{sm})$ is the restriction
of $i$.  We call that element the \emph{fundamental class} of $K$ and denote it $[K]$.

If $\xymatrix{K' \ar[r]^a & K \ar[r]^b & S}$ is a chain of closed subspaces
then the map
\[ a_! (b \circ a)^!\C_S = a_! a^! b^! \C_S \to b^! \C_S \]
lets us interpret the fundamental class of $K'$ as an element of $H^*(K,
b^!\C_S)$.  We will sometimes write $[K']_K$ to emphasize that we mean the image
of $[K']$ in $H^*(K, i^!\C_S)$, or $[K]_S$ to denote the image of $[K]$ in
$H^*(S)$.
\end{definition}

\begin{example}
\label{chern}
We can use \cref{resLemma} to recover a weak version of the degeneracy locus
formula for the Chern classes of a globally generated vector bundle.  Let $X$
be the space of $n \times n$ matrices and let $G = GL_n$.  For $d=1, \ldots, n$ 
let $U_d \subset X$ be the space of matrices whose top $d$ rows have rank $d$, 
and let $Z_d$ be the complement of $U_d$. We will show 

\begin{proposition}
\label{exProp}
$H^*(X/G) \cong H^*(BGL_n)$ is generated as a $\C$-algebra by the images of $[Z_d/G]$,
$d=1,\ldots,n$.
\end{proposition}

This is a weakening of the degeneracy locus formula for the Chern classes $c_i$ in the following
sense: Let $E \to B$ be a rank-$n$ vector bundle on $B$, and suppose that it
admits $n$ global sections $e_i$.  The $e_i$ define a factorization of $B \to
BG$ through some $f: B \to X/G$. If $f$ was smooth then $f^*[Z_d/G]$ is
$[W_d]$, where $W_d = f\i\left( Z_d/G \right) \subset B$.  Thus \cref{exProp}
implies that the cohomological invariants of $E$ are generated by the classes
$[W_d] \in H^*(B)$.  On the other hand, $W_d$ is the locus where $e_1, \ldots,
e_{n-d+1}$ have rank less than $n-d+1$.  The degeneracy locus formula says that
$[W_d] = c_d(E)$, identifying the classes $[W_d]$ with a \emph{particular} set of
generators.

\begin{proof}
Write $[Z_d/G]_{X/G}$ for the image of $[Z_d/G]$ in $H^*(X/G)$.  Let $g_m: GL_m
\into GL_n$ be the map 
\[g_m(M) = 
	\left[ \begin{array}{cc} 
		I_{n-m} & 0 \\ 
		0 & M 
	\end{array} \right] 
\]
$V_m := GL_n/g_m(GL_m)$ is a Stiefel manifold.  $H^*(GL_n, \C) =
\Lambda[x_1,\ldots, x_n]$ where $x_d$ is a primitive element of degree $2d-1$,
and $x_d$ is pulled back from $H^*(V_{d-1}, \C)$.  But $GL_n \to V_{d-1}$
factors through $U_d$, so $x_d$ extends to a class $\bar{x}_d \in H^*(U_d,
\C)$.  $Z_d$ has codimension $d$, and the Gysin sequence for the stratification
\[ Z_d \to X \gets U_d \] 
shows that $\res\, \bar{x}_d \sim [Z_d]_{Z} \in H^{2d}(Z_d,i^!\C)$.  $X$ is
contractible, so the residue map is injective, so \cref{resLemma} implies
$\epsilon_2^{1,2d-1}[Z_d/G]_{X/G} \sim x_d$.  This implies the classes
$[Z_d/G]_{X/G}$ generate $H^*(X/G)$, since the \v{C}ech spectral sequence
degenerates on page 2 and the $E_2$ page of the \v{C}ech spectral sequence is
generated by $H_{pr}^*(G)$ \cite{shulman}.  \end{proof}
\end{example}

\section{Application to the wonderful compactification}   
\label{wonderful}

Let $G$ be semisimple and centerless and let
$X$ be the \emph{wonderful compactification} of $G$.  In this section we will
use \cref{resLemma} to compute $\res\, \alpha$ for $\alpha \in H^*_{pr}(G)$.
The point is that \cref{resLemma} lets us work with equivariant cohomology, and
the equivariant cohomology of $X$ is well understood thanks to
\cite{strickland}. 
	
We recall some facts about the wonderful compactification from
\cite{WC}, \cite{strickland}.
Let $\tilde{G}$ be the universal cover of $G$ and 
$\tilde{B} \subset \tilde{G}$ a Borel subgroup.  Let $\lambda$ be a regular weight of
$\tilde{G}$ and let $V^\lambda$ be the corresponding highest weight
representation.  There is a natural inclusion $G \to\pp\, \emo(V)$ and $X$ is
the closure of $G$ under this embedding.  $X$ possesses both a left $G$-action
and a right $G$-action extending the multiplication actions.  We will let $G^2
= G \times G$ act from the left via $(g_1,g_2)\cdot x = g_1xg_2\i$.

Let $B$ be the
image of $\tilde{B}$ in $G$ and let $\Delta = \{\rho_1,\ldots,\rho_l\}$ be a
set of simple positive roots with respect to $B$.  $Z = X-G$ is a normal
crossings divisor with irreducible components $D_\rho$ labeled by the simple
roots.  For $\Gamma \subset \Delta$ let $D_\Gamma = \bigcap_{\rho\in\Gamma}
D_\rho$ (in particular $X=D_\emptyset$).  Then $D_\Gamma$ is a smooth $G^2$-orbit closure.   Let $P_\Gamma \subset G$ be the parabolic subgroup
generated by $B$ together with all root subgroups $G_\rho$ of
roots $\rho \in \Gamma$, and let $P_\Gamma^-$ be the opposite parabolic.  Let
$G_I$ be the adjoint quotient of $P_I$.  Then $D_\Gamma$ is a fiber bundle over
$G/P_\rho \times P^-_\rho \bs G$ with fiber the wonderful compactification of
$G_I$.  In particular $D_\Delta = G/B \times B^-\bs G$ and $D_\Delta/G^2 = BB
\times BB^-$.  

Write $H^*(BB \times BB^-) = \C[\fk{t} \times \fk{t}] = \C[u_1,\ldots,u_l,v_1,\ldots,v_l]$
where $u_i$ is the first chern class of the pullback of $\rho_i$ to $B$ and $v_i$ is the first Chern
class of the pullback of $\rho_i$ to $B^-$.  Let
\[ x_i = u_i-v_i,\hspace{2cm} y_i=u_i+v_i\]
For $\Gamma \subset \Delta$ let
$W_\Gamma  \subset W$ be the subgroup generated by the reflections associated
to the simple roots $\rho \in \Gamma$, and let $x^\Gamma = \prod_{\rho_i\in
\Gamma} x_i$.  $\C[y_1,\ldots,y_l]$ has an action of
$W$ by via the diagonal action on $\fk{t} \times \fk{t}$.  For $\Lambda \subset
\Delta$ let $A_\Lambda\subset
\C[u_1,\ldots,u_l, v_1,\ldots,v_l]$ be the span of the elements of the form
\[x^\Gamma q(x_1,\ldots,x_l) p(y_1,\ldots,y_l)\]
where $p$ is invariant under $W_{\Delta-(\Lambda \cup \Gamma)}$.
Set $A = A_\emptyset$.  For any graded ring
$R$, let $\tilde{R} \subset R$ be the ideal generated by all elements
of strictly positive degree.

\begin{theorem} (Strickland \cite{strickland}) \hspace{1mm}\label{strickThm}
\begin{enumerate}
\item The pullback $H^*(D_\Lambda/G^2) \to H^*(D_\Delta/G^2)$ is an injection with image $A$.   
\item Under this injection, $[D_\Lambda]$ maps to $x^\Lambda$.
\item The inclusion $\C[u,v]^{W\times W} = A_\Lambda^{W \times W} \into A_\Lambda$ corresponds to
pullback along the projection $D_\Lambda/G^2 \to BG^2$.  
\item $H^*(D_\Lambda/G^2) \to H^*(D_\Lambda)$ is surjective with kernel generated by $\tilde{H}^*(BG^2)$.
\end{enumerate}
\end{theorem}

\begin{remark} Strickland treats only the case $\Lambda = \emptyset$,
but their arguments generalize without any extra work.
\end{remark}

%
%
%

\begin{corollary}
\label{presentation}
$H^*(Z/G, i^!\C_{X/G^2})$ is canonically isomorphic to the cokernel of 
\begin{align*}
 \bigoplus_{i<j} A_{ij} \to \bigoplus_k A_k &&
(f_{ij}) \to \left(\sum_{i<k} x_if_{ik} - \sum_{j>k} x_jf_{kj}
\right)
\end{align*}

$H^*(Z, i^!\C)$ is canonically isomorphic to the cokernel of
\[ \bigoplus_{i<j} \left( A_{ij}/\tilde{A}_{ij}^{\,W\times W} \right) \to
\bigoplus_k \left( A_k/\tilde{A}_k^{\,W\times W} \right)\]
given by the same formula.
\end{corollary}
\begin{proof}
Since $H^*(D_\Lambda/G^2, i^!\C_{X/G^2})$ vanishes in odd degrees,
$H^*(Z/G^2, i^!\C_{X/G^2})$ is the cokernel of the map 
\[\varphi: \bigoplus_{i<j} H^*(D_{\{i,j\}}/G^2, i^!\C_{X/G^2}) \to \bigoplus_k
H^*(D_j/G^2, i^!\C_{X/G^2})\]
where $f \in H^*(D_{\{i,j\}}/G^2, i^!\C_{X/G^2})$ is sent to its image in
$H^*(D_{\{j\}}/G^2, i^!\C_{X/G^2})$ minus its image in $H^*(D_{\{i\}}/G^2,
i^!\C_{X/G^2})$.  

We need only to describe $\varphi$ in terms of the 
$A_\Gamma$.  The isomorphism $H^*(D_\Lambda) \to A_\Lambda$ is
obtained by pulling back all the way to $D_\Delta$, which factors through
pullback to $D_\Lambda$.  The composition
\[H^{d-2}(D_{\{ij\}}/G^2)\cong H^{d}(D_{\{ij\}}/G^2,i^!\C)\to H^{d}(D_{\{j\}}/G^2) \to
H^d(D_{\{ij\}}/G^2)\]
is multiplication by $[D_{\{i\}}/G^2]=x_i$ which implies the first claim.

The argument for $H^*(D_{\{j\}},i^!\C_X)$ is similar, using parts (3) and (4) of
\cref{strickThm} to describe $H^*(D_{\{ij\}})$ and $H^*(D_{\{k\}})$.
\end{proof}

In particular there is a canonical surjection $\psi:
\bigoplus_k A_{\{k\}} \to H^*(Z,i^!\C_X)$ and we know its kernel.  Let $\alpha \in H^{2d-1}_{pr}(G)$
and let $p(u_1,\ldots,u_l) \in \C[u_1,\ldots,u_l]^{W}$ be irreducible and
homogeneous of degree $d$.
Let $\beta = p(x_1+y_1, \ldots, x_l+y_l) - (-1)^d p(x_1-y_1,\ldots,x_l-y_l)$.

\begin{proposition}\hspace{2pt}
\label{finale}
\begin{enumerate}
\item $\beta = \sum_{k=1}^l x_kf_k(x,y)$ for some $f_k(x,y) \in
A_{\{k\}}$.
\item  Suppose that $G$ is not of type $D$ or that $2d \neq l$.
Then $\psi((f_k))\sim \res \alpha$.
\end{enumerate}
\end{proposition}
\begin{proof}

Note that $\beta \sim p(u_1,\ldots,u_l) \pm p(v_1,\ldots,v_l) \in \C[\fk{t}\times
\fk{t}]^{W \times W}$.  Therefore it descends to $BG \times BG$, and lies in
$A$.  On the other hand every term of $\beta$ is divisible by some $x_i$.  This
proves the first claim.

In fact the two summands $p(u)$ and $p(v)$ are individually $W \times
W$-invariant, and therefore descend to $\ol{p(u)} \in B(G \times 1)$ and
$\ol{p(v)} \in B(1 \times G)$
respectively.  This implies that $p(v)$ vanishes under pullback along $X/G =
X/(G \times 1) \to X/G^2$.  Thus the pullback of $\beta$
to $\beta' \in H^*(X/G)$ is a scalar multiple of the pullback of $p(u)$.

By \cref{edgeProp} and the remark following,
$\epsilon_2^{1,2d-1}(\ol{p(u)})\sim \alpha$.  By \cref{resLemma}, in order to
compute $\res \alpha$ up to a scalar, it would suffice to express $\beta'$ as
the image of some $\gamma' \in H^*(Z/G,i^!\C)$ and compute the pullback of
$\gamma'$ to $H^*(Z,i^!\C)$.

The expression $\beta = \sum_{k=1}^l x_kf_k(x,y)$ says that $\beta$ is the image of
$\gamma \in H^*(Z/G^2,i^!\C)$, where $\gamma$ is the element represented by
$(f_k) \in \bigoplus_k A_{\{k\}}$.  The pullback of $\gamma$ to $H^*(Z,i^!\C)$
is equal to the pullback of $\gamma'$, so this completes the proof.
\end{proof}

\begin{example}
Let $G=PGL_2$ and $d=2$.  Then
\begin{enumerate}
\item  $X = \pp Mat_{2 \times 2} \cong \pp^3$, where $Mat_{2\times 2}$
denotes the vector space of $2 \times 2$ matrices. 
\item There is a single simple root $\rho$, and $D_\rho = Z \cong \pp^1
\times \pp^1 \subset X$ is the Segre variety.  
\item $\C[\fk{t}]=\C[u_1]$ is a polynomial ring in one variable.
\end{enumerate}

Now let $\alpha \in H^3_{pr}(G)$ and let $p(u_1) = u_1^2$.  This is an
irreducible $W$-invariant polynomial.  Then $\beta = (x_1+y_1)^2 -
(x_1-y_1)^2  = 4x_1y_1$, so $f_1 = 4y_1 = 4(u_1+v_1) \in H^*(Z/G^2)$.
\Cref{finale} says that to
compute $\res \alpha$ we need only pull back $f_1$ to $H^*(Z)$.  

Let $\sigma \in H^*(\pp^1)$ be the hyperplane section.  Under the surjection $H^*(Z/G^2) \to H^*(Z) = H^*(\pp^1) \otimes
H^*(\pp^1)$, $u_1$ is sent to $\sigma \otimes 1$ while $v_1$ is sent to $1
\otimes -\sigma$.  (The apparent asymmetry arises because the second factor of $G
\times G$ acts via $(1,g)\cdot x \to xg\i$.)  So
$\res\,\alpha \sim \left[\{0\} \times \pp^1 \right] - \left[\pp^1 \times
\{0\}\right]$.
\end{example}

\section{Proof of \cref{resLemma}}
\label{proof}

\begin{notation}
By a \emph{bicomplex} we mean an array $\calL = \calL^{\bul,\bul}$ of objects with
anticommuting morphisms $\partial_I: \calL^{p,q} \to \calL^{p+1,q}$ and $\partial_{II}:
\calL^{p,q} \to \calL^{p,q+1}$, and with $\calL^{p,q}=0$ for $p\ll 0$ or $q \ll 0$.
$\calL[i,j]$ is the shifted bicomplex $\calL[i,j]^{p,q} =
\calL^{p+i,q+j}$. $E_n^{p,q}(\calL)$ is the spectral sequence associated to
$\calL$ with the $\partial_{II}$ orientation, so in particular $E_1^{p,q}(\calL) =
H^q(L^{p,\bul})$.  $\tot(\calL)$ is the total complex and $\H(\calL)$ is its 
cohomology.  Write $F_{n+1} \H(\calL) \subseteq F_n
\H(\calL) \subseteq \ldots \subseteq F_0\H(\calL) = \H(\calL)$ for
the filtration associated to the $\partial_{II}$ orientation and
$\epsilon_n^{p,q}(\calL): F_{p}\H^{p+q}(\calL) \to E_n^{p,q}(\calL)$ for the edge
map.  \end{notation}

Let $\fk{A}$ be an abelian
category and let $M$ be a left-exact additive functor from $\fk{A}$ to bounded-below
complexes of abelian groups.  If $\calA$ is a bounded-below complex in $\fk{A}$
one obtains a first-quadrant bicomplex of abelian groups $A = M(\calA)$ by the
rule $A^{p,q} = M(\calA^q)^p$, with $\partial_I: A^{p,q} \to A^{p+1,q}$ the
differential of $M(\calA^q)$ and $\partial_{II}: A^{p,q} \to A^{p,q+1}$ given
by $(-1)^p M(\partial_\calA)_p$ (that is, up to a sign the degree-$p$ component of the chain complex
map $M(\calA^q) \to M(\calA^{q+1})$ induced functorially by the differential of
$\calA$).

Suppose $\calA^\bul$ and $\calB^\bul$ are bounded-below complexes of 
objects from $\fk{A}$.  Let $\bar{f}: \calA \to \calB$ be a map and $\calC$ be its cone.
Name the natural maps as shown:
\[\xymatrix{ 
	\calA \ar[r]^{\bar{f}} & \calB \ar[r]^{\bar{g}} & \calC
	\ar[r]^{\bar{k}} & \calA[1]
}\]

Let $A = M(\calA)$, $B = M(\calB)$ and $C = M(\calC)$.  Then $\bar{f}$,
$\bar{g}$ and $\bar{k}$ induce maps of bicomplexes
\[\xymatrix{ 
	A \ar[r]^{f} & B \ar[r]^{g} & C \ar[r]^{k} & A[0,1]
}\]

Let $P$ be the bicomplex   
\[P^{r,p} = \left\{\begin{array}{cc}
		E_1^{p,q}(A) & r = 3m \\
		E_1^{p,q}(B) & r = 3m+1 \\
		E_1^{p,q}(C) & r = 3m+2 \\
\end{array} \right. \]
with differential $\partial^P_{II}$ inherited from the differentials $(-1)^p\partial_I$ of
$A$, $B$ and $C$, and $\partial^P_{I}$ identical to $f,g$ or $k$. 
On $E_2(P)$ we obtain a differential
\[-\phi: 
\frac{\ker: E_2^{p+1,q}(B) \to E_2^{p+1,q}(C)}{\im: E_2^{p+1,q}(A)\to
E_2^{p+1,q}(B)}
\to 
\frac{\ker: E_2^{p,q+1}(A) \to E_2^{p,q+1}(B)}{\im: E_2^{p,q}(C)\to
E_2^{p,q+1}(A)}
\]
We will give a more concrete description of $\phi$ in \cref{coneLemma}.

Let $\tilde{F}_p\H(A)$ be the preimage under $f$ of $F_p\H(B)$.  Naturality
of $\epsilon_n^{p,q}(-)$ implies

\begin{lemma} \leavevmode
	\begin{enumerate}
	\item The image of $\left(\tilde{F}_{p+1} \H^{p+q+1}(A) \cap
		F_p \H^{p+q+1}(A)\right)$ under $\epsilon_n^{p,q+1}$ is contained in
		$\ker: 	E_2^{p,q+1}(A) \to E_2^{p,q+1}(B)$.  In
		particular, $\epsilon_n^{p,q+1}$ induces a map 
		\[u_n^{p,q+1}: F_p \H^{p+q+1}(A) \cap \tilde{F}_{p+1}
		\H^{p+q+1}(A) \to 
	        \frac{\ker: E_2^{p,q+1}(A) \to E_2^{p,q+1}(B)}{\im:
E_2^{p,q}(C)\to E_2^{p,q+1}(A)} \]
	
	\item The image of $\left(\tilde{F}_{p+1} \H^{p+q+1}(A) \right)$
		under $\epsilon_n^{p+1,q}(B) \circ f$ is contained in
		$\ker: 	E_2^{p+1,q}(B) \to E_2^{p+1,q}(C)$.  In
		particular, $\epsilon_n^{p+1,q}$ induces a map
		\[v_n^{p+1,q}: f\left[\tilde{F}_{p+1} \H^{p+q+1}(A) \right] \to
                \frac{\ker: E_2^{p+1,q}(B) \to E_2^{p+1,q}(C)}{\im:
E_2^{p+1,q}(A)\to E_2^{p+1,q}(B)}\]

\end{enumerate}
\end{lemma}
%

\begin{lemma}
\label{coneLemma}
The diagram commutes: 
	\[\xymatrix{
		\tilde{F}_{1}\H^{q+1}(A) \ar[r]^-f \ar[d]^{u_2^{0,q+1}} &
		f\left[\tilde{F}_{1}
		\H^{q+1}(A)\right]\ar[d]^{v_2^{1,q}} \\
		E_2^{0,3q+3}(P) & E_2^{1,3q+1}(P) \ar[l]_\phi
	}\]
\end{lemma}
\begin{proof}
If $x \in A^{i,j}$, write $[x]_1$ for the class that $x$ represents in
$E^{i,j}_1(A)$, $[x]_2$ for the class it represents in $E^{i,j}_2(A)$,
and $[x]_3$ for the class it represents in $E_1^{i,3j}(P)$ (assuming these
classes are well-defined).  Similarly for $x\in B^{ij}$ or $x \in
C^{ij}$.  Write $x =_i y$ for $[x]_i = [y]_i$.

Let $a \in \tilde{F}_{1}\H^{q+1}(A)$, and let $(a^{ij})\in
\bigoplus_{i+j=n} A^{i,j}$ be a representative of $a$ with $a^{i,j}=0$ for
$i<0$.  Then
$\epsilon_2^{0,q+1}(a)=[\alpha^{0,q+1}]_2$.
We have $f(a) \in F_{1}\H(B)$ so 
\[0 = \epsilon_1^{0,q+1}f(a) = [f(a^{0,q+1})]_1\]
This implies that there exists $b \in \calB^{0,q}$ with
$\partial_{II}b = f(a^{0,q+1})$.  Then $f(a)$ is represented by
$f(a^{ij}) - (\partial_I + \partial_{II})b$, and
\[\epsilon_2^{1,q}f(a)=[f(a^{1,q})-\partial_I b]_2\]  

Thus our goal is to show $\phi [f(a^{1,q})-\partial_Ib]_3 =
[a^{0,q+1}]_3$.  We can compute $\phi [f(a^{1,q})-\partial_Ib]_3$ by lifting
$g[f(a^{1,q})-\partial_Ib]_1$ to some $\sigma \in E_1^{0,q}(C)$; then
$[k(\sigma)]_3 = \phi [f(a^{1,q})-\partial_Ib]_3$.

We will prove the lemma by brute computation.  Note that $C^{p,\bul}$ is
identical to the cone of $A^{p,\bul} \to B^{p,\bul}$ and the maps $g$ and $k$
are the natural maps to and from the cone.  Explicitly
\begin{enumerate}[label=(\arabic*)]
	\item $C^{p,q} = A^{p,q+1} + B^{p,q}$ with differentials given
by
\begin{align*}
	\partial_{II}(a,b) = (-\partial_{II}a, (-1)^pf(a) + \partial_{II}b)
	&& 
	\partial_I(a,b) = (\partial_Ia, \partial_Ib)
\end{align*}

\item $g: B \to C$ and $k: C \to A[0,1]$ are given by
\begin{align*}
	g(b) = (0,b) 
	&& 
	k(a,b) = a 
\end{align*}
\end{enumerate}

We claim that $\sigma = [(a^{0,q+1}, -b)]_1$ is a lift of
$g[f(a^{1,q})-\partial_I b]_1$.  We must check two
things:

\begin{enumerate}[label=(\alph*)]
	\item $(a^{0,q+1}, -b)$ is closed under the horizontal
	differential:  By fact (1), 
	\[\partial_{II} (a^{0,q+1}, b) = 
	(\partial_{II} a^{0,q+1}, f(a^{0,q+1})-\partial_{II}b)\]
	Since $(a^{ij})$ was a cocycle for $\tot \calA$ and is concentrated on the $j\geq 0$
	rows, $\partial_{II} a^{0,q+1} = 0$.  By choice of $b$ we have
	$\partial_{II}b = f(a^{0,q+1})$.

\item \textbf{$[\partial_I \sigma]_1 = g[f(a^{1,q})-\partial_I b]_1$:}  
	Write $\gamma = g(f(a^{1,q})-\partial_I b)$. By fact (2), $\gamma= \left(0,
	f(a^{1,q})-\partial_I b \right)$.  By (1), $\partial_{II}\left(
	a^{1,q}, 0 \right) = \left(-\partial_{II}a^{1,q},
	-f(a^{1,q})\right)$.  Then 
	\begin{align*}
		\gamma 
		&= \left( 0, f(a^{1,q})-\partial_I b \right) \\ 
		&=_1 \left(0,f(a^{0,q})-\partial_I b \right) +
		   	\partial_{II}\left(a^{1,q}, 0\right) \\ 
		&=  \left(0, f(a^{1,q})-\partial_I b \right) +
			\left(-\partial_{II}a^{1,q}, -f(a^{1,q})
		\right) \\ 
		&=  \left(-\partial_{II} a^{1,q},-\partial_I b
		\right)
	\end{align*}
Since $a$ was a cocycle, $-\partial_{II} a^{1,q} = \partial_I a^{0,q+1}$.  Then
$ \gamma =_1 \left(\partial_I a^{0,q+1}, -\partial_I b \right) = \partial_I
\sigma$.
\end{enumerate}

Finally, fact (2) says $k(\sigma) = [a^{0,q+1}]_1$, which proves the claim.
\end{proof}

Recall the construction of $\vC(X_\bul, \calF)$ from
\cite[\href{http://stacks.math.columbia.edu/tag/06X2}{Tag
06X2}]{stacks-project}.  
Let $M$ be the functor from sheaves of abelian groups on $X/G$ to complexes of
abelian groups, sending $\calG$ to 
\[\ldots \to \Gamma(X_p, \pi_p\i\calG) \to \Gamma(X_{p+1}, \pi_{p+1}\i \calG)
\to \ldots\]
Then $\vC(X_\bul, \calF)$ is the spectral sequence associated to the double
complex $M(\calI)$, where $\calI^\bul$ is an injective resolution of $\calF$;
from the first page it no longer depends on the choice of $\calI^\bul$. 

\begin{theorem}
\label{resLemma}
Let $q>0$, and let $s: X/G \to BG$ be the quotient of $X \to pt$.  Then the diagram below commutes:
\[\xymatrix{
H^{q+1}(Z/G, i^!\C_{X/G}) \ar[d]  & H^{q+1}(X/G) \ar[l] &
H^{q+1}(BG)\ar[d]^{\epsilon^{1,q}_2} \ar[l]_{s^*} \\
H^{q+1}(Z,   i^!\C_X)         &                    & H^{q}_{\mathrm{pr}}(G) \ar[ll]^{\res}
}\]
\end{theorem}
\begin{proof}

The edge maps $\epsilon_2^{0,q+1}$ are natural transformations, and
$\epsilon_2^{0,q+1}(BG_\bul): H^{q+1}(BG) \to H^{q+1}(pt)$ is the zero map, so 
$H^{q+1}(BG) \to H^{q+1}(X/G)$ factors
through $F_1H^{q+1}(X/G)$.  Therefore it would suffice to show commutativity of
\[\xymatrix{
\tilde{F}_1H^{q+1}(Z/G, I^!\C) \ar[d]  & F_1H^{q+1}(X/G) \ar[l]_-\cong &
H^{q+1}(BG)\ar[d]^{\epsilon^{1,q}_2} \ar[l]_-{s^*} \\
H^{q+1}(Z,   i^!\C)         &                           &
H^{q}_{\mathrm{pr}}(G) \ar[ll]^-{\res}
}\]

Let $\calA$ be a complex of injectives quasi-isomorphic to
$i_*i^!\C_{X/G}$ and $\calB$ be a complex of injectives quasi-isomorphic
to $\C_{X/G}$.  Let $\bar{f}: \calA \to \calB$ be the natural map
$i_*i^!\C_{X/G} \to \C_{X/G}$ and let $\calC$ be its cone.  Then $\calC$ is a
complex of injectives quasi-isomorphic to $j_*j\i\C_{X/G}$.  We will apply 
\cref{coneLemma} to $\calA \to \calB \to \calC$.  Note that in this case
$E_2^{i,j}(C)=0$ unless $i=j=0$.  In particular the domain of $\phi$ is a
quotient of $\vC_2^{1,q}(X_\bul, \C_X)$ and the codomain of
$\phi$ is a subspace of $\vC_0^{0,q+1}(Z_\bul, i^!\C_X) \subset H^{q+1}(Z,
i^!\C_X)$.  The commuting diagram of \cref{coneLemma} implies
commutativity of the simpler diagram
\[\xymatrix{
\tilde{F}_1 H^{q+1}(Z/G, i^!\C_{X/G}) \ar[r] \ar[d]  
  &F_1 H^{q+1}(X/G)                      \ar[d] \\
H^{q+1}(Z, i^!\C_X) 
 &  \ar[l]^{\varphi}  \vC^{1,q}_2(X_\bul, \C) \ar[l]
}\]
where we have written $\varphi$ for the composition 
\[ 
\vC^{1,q}_2(X_\bul, \C)\to \frac{\vC^{1,q}_2(X_\bul, \C)}{\vC_2^{1,q}(Z_\bul, i^!\C_X)} 
\xymatrix{\ar[r]^\phi &}  \ker\left[ \vC_2^{0,q+1}(Z_\bul, i^!\C_X) \to
\vC_2^{0,q+1}(X_\bul, \C_X)\right] \into H^{q+1}(Z, i^!\C_X)\]
 It suffices to show that
$\res: H^{q}_{\mathrm{pr}}(G) \to H^{q+1}(Z,   i^!\C)$ is equal to the
composition
\[ \xymatrix{H^{q}_{\mathrm{pr}}(G) \ar[r]^-= &
\vC_2^{1,q}(BG_\bul, \C) \ar[r]^-{s^*_1} & 
\vC_2^{1,q}(X_\bul, \C) \ar[r]^-\varphi &
\vC_2^{0,q+1}(X_\bul, i_*i^! \C) }
\]
Here $s_\bul: X_\bul \to BG_\bul$ is the map induced by $s$, so in particular $s_1: G \times
X \to G$ is the projection.

Let $\alpha \in H^{q+1}_{pr}(G)$.  We will compute $\varphi (s_1^*\alpha)$
referring to the diagram below: 
\[\xymatrix{
                             &				&			                      &					  &                     \\
H^q(G \times Z, i^!\C) \ar[r]\ar[u] & H^q(G \times X) \ar[r]^{j_1^*} \ar[u] &
H^q(G \times G) \ar[r] \ar[u]               & H^{q+1}(G \times Z, i^!\C) \ar[r]
\ar[u] & H^{q+1}(G \times X) \ar[u] \\
H^q(Z, i^!\C) \ar[r]\ar[u] & H^q(X) \ar[r]         \ar[u] & H^q(G) \ar[r]^\res
\ar[u]^{\partial} & H^{q+1}(Z, i^!\C) \ar[r] \ar[u] & H^{q+1}(X) \ar[u] \\
0 \ar[u]                     &  0 \ar[u]                      & 0 \ar[u]                              & 0 \ar[u]                          & 0 \ar[u] 
}\]
The class $\varphi(s_1^*\alpha)$ is represented by any $\beta \in H^{q+1}(Z,
i^!\C)$ that lifts to $\sigma \in H^q(G)$ with $\partial\sigma =
j_1^*s_1^*\alpha$, so we need to show that $\beta = \res\, \alpha$ works.  Then
$\beta$ lifts to $\alpha \in H^q(G)$, and $\partial \alpha = \mu^*\alpha -
\pi_1^*\alpha$, where $\mu: G \times G \to G$ is the multiplication and $\pi_0,
\pi_1: G \times G \to G$ are the projections.  Since $\alpha$ is primitive
$\mu^*\alpha = \pi_0^*\alpha + \pi_1^*\alpha$ and $\partial \alpha =
\pi_0^*\alpha$.  As $s_1 \circ j_1 = \pi_0$ this proves the claim.  
\end{proof}

\section{Computation of the edge map for $BG$}\label{edgeCalc}

Here we explain how to compute the edge map $\epsilon_2^{1,q}:
H^{q+1}(BG) \to H^q(G)$ coming from the \v{C}ech spectral sequence
$\vC^{p,q}_n(BG_\bul, \C)$.  This is a
repackaging of work of Bott \cite{bott}.  We assume that $G$ is connected and semisimple.
Let $\fk{g}$ be the Lie algebra of $G$.  Our explicit models of $H^*(BG)$ and
$H^*(G)$ are respectively the ring of invariant
polynomials $\C[\fk{g}]^G = \left( \sym^\bul \fk{g}^\vee[2] \right)^G$ and the
cohomology of the Chevalley-Eilenberg complex
$\fk{C}(\fk{g}^\vee)$ of $\fk{g}$.  The precise identifications are explained
in \cref{tauDef} and \cref{models}.

\begin{definition}
Denote by $\Omega^q_X$ the sheaf of smooth complex-valued $q$-forms on
$X$.  The \emph{\v{C}ech-de Rham} bicomplex $\dR^{\bul,\bul}_0(X_\bul, \C)$ is the
bicomplex with $\dR^{p,q}_0(X_\bul, \C) = \Gamma(X_p, \Omega^q)$.  The
differential $\partial_{II}$ is $(-1)^p$ times the de Rham differential
of $X_p$, and the differential $\partial_I$ is the alternating sum of
pullbacks along the projections $X_{p+1} \to X_p$.  Write
$\dR_\bul(X_\bul, \C)$ for the associated spectral sequence. 
\end{definition}

\begin{notation}
If $M^{p,q}$ is a cosimplicial complex (where $p$ is the ``simplicial
coordinate'' and $q$ is the ``chain complex coordinate''), with differential
$\delta: M^{p,q} \to M^{p,q+1}$ and coface maps $\phi^{p,q}_i: M^{p,q} \to
M^{p+1,q}$, $i=0,1,\ldots,p$,  let $\kk M$ be the associated (anticommuting)
bicomplex defined by $(\kk M)^{p,q} = M^{p,q}$ with differentials 
\[
\partial^{p,q}_I(m) = \sum_{i=0}^p (-1)^i \phi_i^{p,q}(m) \hspace{10mm}
\partial^{p,q}_{II}(m) = (-1)^p \delta(m)  
\] 
Write $H^*M$ for the cohomology of
the total complex $\tot (\kk M)$ of $\kk M$.  We identify cosimplicial modules
with cosimplicial complexes concentrated in degree $0$. 

 Let $\Gamma_\Delta(X_\bul, -)$ be the functor sending sheaves of abelian groups on
$X_\bul$ to cosimplicial modules by the rule $\Gamma_\Delta(X_\bul,\calF)_p =
\Gamma_\Delta(X_p, \calF)$.  
\end{notation}

Since Bott's results use the \v{C}ech-de Rham spectral sequence rather than the
\v{C}ech spectral sequence, we need to compare them.  We will do so via a
third spectral sequence studied by Friedlander 
\cite[Proposition 2.4]{friedlander}.  The following lemma and corollary
are probably well known.

\begin{lemma}\label{simpCoh}
Let $\calI$ be an injective resolution of $\C_{X/G}$.  There is a bicomplex
$Q^{p,q}$ with maps 
$\vC_0(X_\bul, \calI) \to Q  \gets \dR_0(X_\bul, \C)$
inducing isomorphisms $\vC_n(X_\bul, \calI) \to E_n(Q) \gets
\dR_n(X_\bul, \C)$ for all $n \geq 1$. 
 \end{lemma}
\begin{proof}
Following \cite{friedlander}, we work with sheaves on the simplicial
scheme $X_\bul$. For $q$ fixed, the collection of sheaves $\Omega^q_{X_p}$
defines such a sheaf.  $\Omega^\bul_{X_\bul}$ defines a complex of
sheaves on $X_\bul$ resolving the constant sheaf $\C_{X_\bul}$.  Similarly $\pi_p\i
\calI^q$ defines a complex of sheaves on $X_\bul$ resolving $\C_{X_\bul}$.

Let $J^\bul$ be an injective complex of sheaves on $X_\bul$ resolving
$\C_{X_\bul}$, so that $\Gamma_\Delta(X_\bul, J)$ defines a cosimplicial
complex.  Let $Q^{p,q} = \kk \Gamma_\Delta(X_\bul, J)$.  The identity morphism
of $\C_{X_\bul}$ induces maps of resolutions $x': \Omega^{\bul}_{X_\bul} \to J$
and $y': \pi_\bul\i \calI^\bul \to J$, and taking global sections in each
simplicial degree we obtain $x$ and $y$. 

Each $J^\bul_p$ is an injective resolution of $X_{\C_p}$ (see the first
paragraph of the proof of \cite[Proposition 2.4]{friedlander}).
$\pi_p\i\calI^\bul$ is an injective resolution of $\C_X$ (see the proof of
\cite[\href{http://stacks.math.columbia.edu/tag/06XF}{Tag 06XF}]{stacks-project}) and
$\Omega^\bul_{X_p}$ is a $\Gamma(X_p, -)$-acyclic resolution of $\C_X$, so $x'$
and $y'$ induce quasi-isomorphisms $\Gamma(X_p, \pi_p\i\calI) \to \Gamma(X_p,
J_p)$ and $\Gamma(X_p, \Omega_p^\bul) \to \Gamma(X_p, J_p)$.  In other words,
$x$ and $y$ induce isomorphisms on $E_1$.  Therefore they induce isomorphisms
on all later pages.
\end{proof}

\begin{corollary}
The edge maps $\epsilon_n^{p,q}: F_pH^{p+q}(X/G, \C) \to H^q(X_p, \C)$
obtained from the \v{C}ech spectral sequence and the \v{C}ech-de Rham
spectral sequence are identical for $n \geq 1$.
\end{corollary}

This reduces the problem to computing the edge map for the \v{C}ech-de Rham
spectral sequence. 

\begin{notation}  
If $M$ is a module, let $CM$ be the cosimplicial module $C^p
M = M^{\oplus p+1}$, where the coface maps are defined by inserting
zeros (for example $(m_0, \ldots, m_p) \to (m_0, \ldots, m_i, 0,
m_{i+1}, \ldots, m_p)$).  Let $\Sigma M \subset CM$ be the
sub-cosimplicial module defined by $\Sigma^p M = \{(m_0, \ldots, m_p) |
\sum_i m_i = 0\}$.

Let $\delta: \fk{g}^\vee \to \wedge^2 \fk{g}^\vee$ be the dual of the Lie
bracket.  This extends using the Leibniz rule to $\delta: \wedge^d \fk{g}^\vee
\to \wedge^{d+1} \fk{g}^\vee$.  The \emph{Chevalley-Eilenberg} complex
$\fk{C}\fk{g}^\vee$ of $\fk{g}$ is defined by $\fk{C}^d \fk{g}^\vee =
\wedge^d\fk{g}^\vee$ with differential $\delta$.  We consider $\fk{C}$ to be a
covariant functor from Lie coalgebras to chain complexes.
\end{notation}

There is a natural cosimplicial Lie coalgebra structure on $\Sigma
\fk{g}^\vee$.  Applying $\fk{C}$ we obtain a cosimplicial chain complex
$\fk{C}\left(\Sigma\fk{g}^\vee\right)$.  

\begin{definition} \label{tauDef}
Identifying elements of
$\fk{C}^q(\Sigma^p \fk{g}^\vee )$ with left-invariant $q$-forms on $BG_p$
\cite[decomposition lemma]{bott}
defines a map of cosimplicial chain complexes $\tau: \fk{C}^\bul(\Sigma^\bul
\fk{g}^\vee) \to \dR(BG_\bul, \C)$.
\end{definition}

\begin{lemma} \label{comparison}
	(Bott) $\tau$ induces an isomorphism $H^*
	\fk{C}(\Sigma\fk{g}^\vee) \to \H^*\dR(BG_\bul, \C)$ 
\end{lemma}

\begin{remark} 
\Cref{comparison} is the decomposition lemma of \cite{bott}
plus the discussion immediately afterwards.  Bott works with the restriction of
$\tau$ to the $G$-invariants of
$\fk{C}(\Sigma\fk{g}^\vee)$.  Since $G$ is semisimple and connected $H^*
\fk{C}(\Sigma\fk{g}^\vee)^G \to H^* \fk{C}(\Sigma\fk{g}^\vee)$ is
an isomorphism, and it will be convenient for us to use this
less-refined version of Bott's decomposition lemma.  
\end{remark}

Bott then studies the spectral sequence of $\kk \fk{C}(\Sigma \fk{g}^\vee)$ in
the $\partial_I$ orientation (rather than the $\partial_{II}$ orientation used
to get the \v{C}ech-de Rham spectral sequence).  He shows it degenerates on
page 1, and (using the Eilenberg-Zilbur theorem) shows that the
Alexander-Whitney map identifies its total cohomology with $\C[\fk{g}]^G$ 
\cite[Lemma 3.1]{bott}.  Together with \cref{comparison} this
recovers the Chern-Weil isomorphism.

\begin{theorem} \label{models} \leavevmode
	\begin{enumerate}
		\item (Bott, \cite[Theorem 1]{bott}) $\tau$ induces an isomorphism
			$\fk{C}[\fk{g}]^G \to H^*(BG)$
		\item (Chevalley-Eilenberg, \cite[Theorem 15.2]{chevEilenberg})
			$\tau^{1,\bul}$ induces an isomorphism $\fk{C}\fk{g}^\vee \to H^*(G)$. 
	\end{enumerate}
\end{theorem}

\begin{remark} To get \cref{models} both Bott and
	Chevalley-Eilenberg require that $G$ be compact.  Since $G$ is
	semisimple we can reduce to that case by choosing a maximal compact $K
	\subset G$, whose Lie algebra $\fk{k}$ is a real form of $\fk{g}$ (see
	\cite{Knapp} 3.4, 3.5 and 4.5).  Then $K \to G$ is a deformation
	retract, $BK \to BG$ is a homotopy equivalence, and the
	Chevalley-Eilenberg complex of $\Sigma \fk{g}$ is the base change to
	$\C$ of the (real) Chevalley-Eilenberg complex of $\fk{k}_\R^\vee :=
	\hom_{\R}(\fk{k}, \R)$.  So if these statements hold for $K$ then
	they hold for $G$ as well.  \end{remark}

We can now describe an algorithm for the edge map.  The input is a homogeneous
invariant polynomial $a \in \left( \sym^d \fk{g}^\vee[2] \right)^G$ of degree
$2d$.  
\begin{algorithm} \leavevmode
\begin{enumerate} 
	\item Identify $a$ with some $\bar{a} \in \bigotimes^{d}
		\fk{g}^\vee$ under the inclusion $\sym^d \fk{g}^\vee \into
		\bigotimes^d \fk{g}^\vee$.  
	\item Use the inverse Alexander-Whitney map (see \cite{weibel} 8.5.4
		for the simplicial version) to identify $\bar{a}$ with some
		$a^{d,d} \in
		\kkk \left( \Sigma \fk{g}^\vee\right)^{\otimes d}$.  
	\item We now construct an element $(a^{p,q}) \in \bigoplus_{p+q=2d}
		\fk{C}^q (\Sigma \fk{g}^\vee)^p$ that represents the class of
		$H^* \fk{C} \left( \Sigma \fk{g}^\vee \right) $ corresponding
		to $a$ under the isomorphism of \cref{models}.
		Set $a^{p,q}=0$ for $p>d$.  We got $a^{d,d}$ in the
		previous step.  Now choose the remaining $a^{p,q}$ to satisfy
		the recurrence $\partial_{II} a^{p,q} = \partial_I
		a_{p-1,q+1}$.  This is solvable since the $\partial_I$-cohomology of $\kk^{\bul, q} 
		\fk{C}(\Sigma^\vee)$ is concentrated in degree $q$.
	\item Finally, $a^{1,2d-1} \in \wedge^{2d-1}\fk{g}^\vee =
		\fk{C}^{2d-1}\fk{g}^\vee$ represents $\epsilon_2^{1,2d-1}(a)$.
\end{enumerate}
\end{algorithm}

As an example we show
\begin{proposition}
	Let $G = PGL_2$, and let $a \in \C[\fk{g}]^G$ be the determinant
	polynomial.  Let $\eta \in \wedge^3 \fk{g}^\vee$ be the form
	$\eta(u,v,w) = \langle [u,v],w \rangle$ where $\langle -,- \rangle$ is
	the Killing form, and let $[\eta] \in H^3(\fk{C}\fk{g}^\vee)$ be the class
	it represents.  Then $\epsilon_2^{1,3}(a)=\frac{1}{2}[\eta]$.
\end{proposition}
\begin{proof} Choose coordinates $x,y,z \in \fk{g}^\vee$ given by  
$\left[ \begin{array}{cc} 
		x & y \\ 
		z & -x
\end{array} \right]$.
In other words $x, y, z$ are dual to the basis $(h, e, f)$ where 
\[ 
h = \left[ \begin{array}{cc} 1&0\\0&-1 \end{array} \right] \hspace{10mm} 
e = \left[ \begin{array}{cc} 0&1\\0&0 \end{array}  \right] \hspace{10mm} 
f = \left[ \begin{array}{cc} 0&0\\1&0 \end{array} \right] \hspace{10mm} \] 
Let $a = -x^2 - yz \in \left(\sym^2 \fk{g}^\vee \right)^G$, the determinant 
polynomial.  Then $\bar{a} = -x \otimes x - \frac{1}{2}\left( y \otimes z + z
\otimes y \right)$.  Under the identification $H^1(\Sigma \fk{g}^\vee) \otimes
H^1(\Sigma \fk{g}^\vee) \cong \fk{g}^\vee \otimes \fk{g}^\vee$, $\bar{a}$ is
represented by \[ \col{-x \\ x} \otimes \col{ x \\ -x} + \frac{1}{2}
\col{ -y \\y} \otimes \col{ z \\ -z} + \frac{1}{2}\col{ -z \\z}
\otimes \col{ y \\-y} \in C^1\fk{g}^\vee \otimes C^1 \fk{g}^\vee\] 

To keep the notation under control, write $x_1 = \col{x \\ 0 \\0}$, $y_2 =
\col{0 \\y \\ 0}$, etc. and identify $u \wedge v = \frac{1}{2} (u\otimes v
-v\otimes u)$.  
Then the inverse Alexander-Whitney map sends this
representative of $\bar{a}$ to $a^{2,2} \in \bigwedge^2 C^2 \fk{g}^\vee$ given by
\begin{align*}
	a^{2,2} = \hspace{2mm}
& 2 \left( x_2 \wedge x_1 + x_3 \wedge x_2 + x_1 \wedge x_3 \right)  \\
& +z_2 \wedge y_1 + y_1 \wedge z_3 + y_3 \wedge z_2 + z_3 \wedge y_3 \\
& +y_2 \wedge z_1 + z_1 \wedge y_3 + z_3 \wedge y_2 + y_3 \wedge z_3 
\end{align*} 

Now we compute the image of $a^{2,2}$ under the Chevalley-Eilenberg differential
$\partial_{II}$.  Note that 
\[ 
	\partial_{II} y = 4 x \wedge y \hspace{1.5cm}
	\partial_{II} z = 4 z \wedge x \hspace{1.5cm}
\partial_h x = 2 z \wedge y \]
Therefore $\partial_h y_3 = 4x_3\wedge y_3$, and similarly for the
other variables.  Writing $x_1y_2z_3 := x_1\wedge y_2 \wedge z_3$ etc. one
obtains by the Leibniz rule 

\begin{align*} \frac{1}{4}\partial_{II} a^{2,2} =& \,\,\,  x_1 y_2 z_2 - x_2 y_1 z_1 +
	x_2 y_3 z_3 -  x_3 y_2 z_2 + x_3 y_1 z_1 -  x_1 y_3 z_3 \\& +z_2 x_2
	y_1 - z_2 x_1 y_1 + x_1 y_1 z_3 - y_1 z_3 x_3 + x_3 y_3 z_2 - y_3 z_2
	x_2  \\& +x_2 y_2 z_1 - y_2 z_1 x_1 +z_1 x_1 y_3 -  z_1 x_3 y_3 + z_3
	x_3 y_2 -  z_3 x_2 y_2  \in \bigwedge^3 C^2 \fk{g}^\vee \end{align*}

In our coordinates $\eta = 8 x \wedge y \wedge z$,
which we identify with the element $8 (x_1-x_2)\wedge (y_1-y_2) \wedge
(z_1-z_2) \in \bigwedge^3 \Sigma^1 \fk{g}^\vee$.  We compute \begin{align*}
	\partial_I \frac{1}{8}\eta =& \,\, (x_1-x_2)\wedge (y_1-y_2) \wedge
	(z_1-z_2) \\ &-     (x_1-x_3)\wedge (y_1-y_3) \wedge (z_1-z_3) \\ &+
	(x_2-x_3)\wedge (y_2-y_3) \wedge (z_2-z_3) \end{align*} Sage
\cite{sage} verifies that this equals $\partial_{II}\frac{1}{4} a^{2,2}$, 
so $a^{1,3} = \frac{1}{2}\eta$ solves the recurrence.
\end{proof}

\bibliographystyle{plain} \bibliography{vinberg}

\end{document}